\documentclass{amsart}
 \usepackage{amssymb, latexsym}
 \usepackage[dvips]{graphics}
\begin{document}
 \title{A Note on Farey Fractions With Odd Denominators}
 \author{Alan Haynes}
 \thanks{This work was done as part of the REU program at
    the University of Illinois at Urbana-Champaign,
    Summer 2001}

 \newtheorem{theorem}{Theorem}
 \newtheorem{lemma}{Lemma}
 \newtheorem{corollary}{Corollary}

  \begin{abstract}
  In this paper we examine the subset
  $\mathcal{F}_{Q, \text{odd}}$
  of Farey fractions of order
  $Q$
  consisting of those fractions whose denominators are odd.
  In particular, we consider the frequencies of the values
  taken on by
  $\Delta =qa'-aq'$ where
  $a/q<a'/q'$
  are consecutive in
  $\mathcal{F}_{Q,\text{odd}}$.
  After proving an asymptotic result for these frequencies,
  we generalize the result to the subset of elements of
  $\mathcal{F}_{Q, \text{odd}}$
  formed by restriction to a subinterval
  $[\alpha , \beta]\subseteq [0,1]$.
  \end{abstract}
 \maketitle

 \section{Introduction}
  For
  $Q\in\mathbb{N}$,
  the Farey fractions of order
  $Q$
  are defined as
  \[\mathcal{F}_Q=\left\{\frac{a}{q}\in\mathbb{Q} :
  0\leq\frac{a}{q}\leq 1, 1\leq q\leq Q,\,\mathop{\mathrm
  {\mathop{\mathrm{GCD}}}}(a,q)=1\right\},\]
  where $\mathop{\mathrm{GCD}}$ denotes the greatest
  common divisor function. Farey fractions have been
  studied in the past mainly for two reasons. First, they
  are important in the field of diophantine approximation.
  Second, there is a connection between Farey fractions
  and the Riemann Hypothesis ([4],[7]). The distribution
  of subsets of $\mathcal{F}_Q$ satisfying congruence
  conditions on the numerators and denominators has not
  been investigated as much, but there is still a significant
  motivation to do so. For example, the analog for Dirichlet
  L-functions of the results of Franel and Landau ([4],[7])
  involve such subsets. It is well known that if
  $a/q<a'/q'$
  are consecutive elements of
  $\mathcal{F}_Q$,
  then
  $qa'-aq'=1$.
  In this paper, we will begin by considering a subset of
  $\mathcal{F}_Q$
  defined as
  \[\mathcal{F}_{Q,\text{odd}}=\left\{\frac{a}{q}\in \mathcal
  {F}_Q : q\text{ odd}\right\} .\]
  In general, if
  $a/q<a'/q'$
  are consecutive elements of
  $\mathcal{F}_{Q,\text{odd}} $,
  the value of
  $\Delta =qa'-aq'$
  need not be one. In view of this we define, for
  $k\in \mathbb{N}$,
  the numbers
  \[N_{Q,\text{odd}}(k)=\#\left\{\frac{a}{q}<\frac{a'}
  {q'}\text{ consecutive in } \mathcal{F}_{Q,\text{odd}} :
  qa'-aq'=k\right\}.\]
  We will show that the asymptotic frequencies
  $\rho _{\text{odd}}(k)$
  defined by
  \[\rho _{\text{odd}}(k)=\lim_{Q\to\infty}
  \frac{N_{Q,\text{odd}}(k)}{\#\mathcal{F}
  _{Q,\text{odd}}}\]
  exist and, further, that they can be computed exactly as
  \[\rho _{\text{odd}}(k)=\frac{4}{k(k+1)(k+2)}.\]
  In the last two sections, we will extend this result to a
  more general set of Farey fractions, consisting of those
  elements of
  $\mathcal{F}_{Q,\text{odd}}$
  which fall in a given subinterval
  $[\alpha , \beta]\subseteq [0,1]$.
  It would be interesting to investigate the following
  questions, which we leave as open problems.
  \begin{enumerate}
  \item Choose two positive integers
        $k_1$
        and
        $k_2$
        and three consecutive elements
        $a/q<a'/q'<a''/q''$
        of
        $\mathcal{F}_{Q,\text{odd}}$.
        In the limit as
        $Q$
        goes to infinity, what is the probability that
        $qa'-aq'=k_1$
        and
        $q'a''-a'q''=k_2$?
  \item What can be deduced about subsets of
        $\mathcal{F}_{Q}$
        with more general congruence conditions on
        the numerators and denominators? Is it possible to
        generalize the above results to include such cases?
  \end{enumerate}

  \section{Preliminary Results}
  Given that
  $a/q<a'/q'<a''/q''$
  are consecutive elements of
  $\mathcal{F}_Q$,
  we will use the following four results.
  \begin{enumerate}
  \item $q+q'>Q$\quad [6, Theorem 30].
  \item At least one of
        $q$ and $q'$
        is odd (follows from the identity
        $qa'-aq'=1$).
  \item There exists an integer
        $r$
        such that
        $q+q''=rq'$
        and
        $a+a''=ra'$\quad [6, Theorem 29].
  \item $q''=\left[ \frac{Q+q}{q'}\right] q'-q$\quad
        (follows from [5, Lemma 1]).
  \end{enumerate}
  For
  $k,Q\in \mathbb{N}$,
  we define the regions
  \[T_{k,Q}=\{ (x,y)\in\mathbb{R}^2 : 0\leq x,y
  \leq Q, x+y>Q, ky\leq Q+x<(k+1)y\}.\]
  Figure 3 in [1] provides a pictorial description of
  these regions, scaled by a factor of Q. To compute
  $\mathop{\mathrm{Area}}(T_{k,Q})$,
  write, for
  $k\in\mathbb{N}$,
  \[P_{1,Q}(k)=\left(\frac{(k-1)Q}{k+1},\frac{2Q}{k+1}
  \right),\quad P_{2,Q}=\left( Q,\frac{2Q}{k+1}\right).\]
  Then we have:
  \begin{eqnarray}
     \mathop{\mathrm{Area}}(T_{k,Q}) &=&\frac{|P
     _{1,Q}(k)-(Q,0)|}{(\sqrt{2} )Q}\times\frac{|P
     _{2,Q}(k-1)-(Q,0)|}{Q}\times\frac{Q^2}{2} \nonumber
  \\ & & -\frac{|P_{1,Q}(k+1)-(Q,0)|}{(\sqrt{2} )Q}\times
     \frac{|P_{2,Q}(k)-(Q,0)|}{Q}\times\frac{Q^2}{2}
     \nonumber
  \\ &=&\frac{4Q^2}{k(k+1)(k+2)}\quad \text{for } k>1,
     \nonumber
  \\ \mathop{\mathrm{Area}}(T_{1,Q}) &=&\frac{Q^2}{6}. \nonumber
  \end{eqnarray}

  \section{Asymptotic Lattice Point Estimates}
  Given a convex region
  $\Omega\in\mathbb{R}^2$
  such that
  $\Omega $
  is contained in the square with vertices
  $(0,0), (Q, 0), (Q, Q),$
  and
  $(0,Q)$,
  define the following numbers.
  \begin{eqnarray}
     M_{\text{odd}}(\Omega ) &=& \#\{ (x,y)\in
     (\mathbb{Z}^2\cap\Omega ) : x \text{ odd}\}, \nonumber
  \\ N(\Omega ) &=& \#\{ (x,y)\in (\mathbb{Z}^2\cap\Omega
     ) : \, \mathop{\mathrm{GCD}}(x,y)=1\}, \nonumber
  \\ N_{\text{odd}}(\Omega) &=& \#\{ (x,y)\in (\mathbb{Z}
     ^2\cap\Omega ) : \, \mathop{\mathrm{GCD}}(x,y)=1, x
     \text{ odd}\}, \nonumber
  \\ N_{\text{odd}, \text{odd}}(\Omega) &=& \#\{ (x,y)\in
     (\mathbb{Z}^2\cap\Omega ): \, \mathop{\mathrm
     {GCD}}(x,y)=1, x \text{ odd}, y \text{ odd}\},
     \nonumber
  \\ N_{\text{odd},\text{even}}(\Omega) &=& \#\{ (x,y)\in
     (\mathbb{Z}^2\cap\Omega ) : \, \mathop{\mathrm
     {GCD}}(x,y)=1,  x \text{ odd}, y \text{ even}\}.
     \nonumber
  \end{eqnarray}
  Note that
  $N_{\text{odd}}(\Omega )=N_{\text{odd},\text
  {odd}}(\Omega )+N_{\text{odd},\text{even}}(\Omega )$.
  Also, define the convex region
  $\Omega ^* =\{(x,y/2) : (x,y)\in\Omega\}$,
  and observe that
  $N_{\text{odd},\text{even}}(\Omega )=N_{\text
  {odd}}(\Omega ^*)$.
  Now we will prove the following result.
  \begin{lemma}
  For any region
  $\Omega\in \mathbb{R}^2$
  as defined above, we have:
  \begin{enumerate}
  \item $N_{\text{odd}}(\Omega )=\frac{4}{\pi ^2}\mathop
        {\mathrm{Area}}(\Omega )+O(Q\log Q).$
  \item $N_{\text{odd}, \text{odd}}(\Omega )=\frac{2}{\pi
        ^2}\mathop{\mathrm{Area}}(\Omega )+O(Q\log Q).$
  \item $N_{\text{odd},\text{even}}(\Omega )=\frac{2}{\pi
        ^2}\mathop{\mathrm{Area}}(\Omega )+O(Q\log Q).$
  \end{enumerate}
  \end{lemma}
  \begin{proof}
  Obviously,
  $M_{\text{odd}}(\Omega )=(1/2)\mathop{\mathrm
  {Area}}(\Omega )+O(Q)$.
  Now
  \begin{eqnarray}
     N_{\text{odd}}(\Omega ) &=& \sum_{\substack{n=1 \\ n\text{ odd}}}^Q
     \mu (n)M_{\text{odd}}\left(\frac{\Omega}{n}\right) \nonumber
  \\ &=& \sum_{\substack{n=1 \\ n\text{ odd}}}^Q \mu (n)\left( \frac
     {\mathop{\mathrm{Area}}(\Omega /n)}{2}+O\left(
     \frac{Q}{n}\right) \right) \nonumber
  \\ &=& \left( \frac{\mathop{\mathrm{Area}}(\Omega)}
     {2}\sum_{\substack{n=1 \\ n\text{ odd}}}^\infty
     \frac{\mu (n)}{n^2}\right) +O(Q\log Q), \nonumber
  \end{eqnarray}
  where
  $\mu (n)$
  is the M\"{o}bius function. Using the Euler product
  formula for the Riemann Zeta function, we obtain the
  formula
  \[\sum_{\substack{n=1 \\ n\text{ odd}}}^\infty\frac
  {\mu (n)}{n^2}=\zeta (2)^{-1}\left( \frac{1}{1-2
  ^{-2}}\right)=\frac{8}{\pi ^2}.\]
  Using this formula in the above sum, we see that
  \[N_{\text{odd}}(\Omega)=\frac{4}{\pi ^2}\mathop{\mathrm
  {Area}}(\Omega )+O(Q\log Q).\]
  The other two parts of the proof follow from the remarks
  preceding the statement of the theorem, and the fact that
  $\mathop{\mathrm{Area}}(\Omega ^*)=1/2\mathop{\mathrm
  {Area}}(\Omega )$:
  \begin{eqnarray}
     N_{\text{odd},\text{even}}(\Omega ) &=& N_{\text
     {odd}}(\Omega ^*)=\frac{2}{\pi
     ^2}\mathop{\mathrm{Area}}(\Omega)+O(Q\log Q), \nonumber
  \\ N_{\text{odd},\text{odd}}(\Omega ) &=& N_{\text
     {odd}}(\Omega )-N_{\text{odd},\text{even}}(\Omega
     )=\frac{2}{\pi ^2}\mathop{\mathrm{Area}}(\Omega )+O(Q\log Q). \nonumber
  \end{eqnarray}
  \end{proof}
  In particular, since
  $\mathcal{F}_{Q,\text{odd}}$
  can be defined as
  \[\mathcal{F}_{Q,\text{odd}}=\left\{ (a,q)\in \mathbb{Z}
  ^2 : 1\leq q\leq Q, 0\leq a\leq q,\mathop{\mathrm{GCD}}(a,q)=1, q \text{
  odd}\right\}\]
  we can apply Lemma 1 to obtain
  \[\#\mathcal{F}_{Q,\text{odd}}=\frac{2Q^2}{\pi ^{2}}
  +O(Q\log Q).\]

  \section{Asymptotic Frequencies in the Interval [0,1]}
  When two consecutive elements
  $a/q<a'/q'$
  of
  $\mathcal{F}_Q$
  both have odd denominators then they will also be
  consecutive in
  $\mathcal{F}_{Q,\text{odd}}$,
  and we will have
  $\Delta =qa'-aq'=1$.
  Hence, if we choose two consecutive elements
  $a/q<a''/q''$
  in
  $\mathcal{F}_{Q,\text{odd}}$
  such that
  $\Delta=qa''-aq''=k>1$,
  we know that these fractions are not consecutive in
  $\mathcal{F}_Q$.
  The fact that two even denominators can not occur
  consecutively in
  $\mathcal{F}_Q$
  tells us that there is exactly one fraction between
  $a/q$
  and
  $a''/q''$;
  call it
  $a'/q'$.
  Now, let
  $r\in\mathbb{Z}$
  be the integer such that
  $q+q''=rq'$
  and
  $a+a''=ra'$.
  Then, by solving the systems of equations
  \begin{eqnarray}
     a'(qa'')-aq'a'' &=& a'', \nonumber
  \\ q'(qa'')-qa'q'' &=& q \nonumber
  \end{eqnarray}
  and
  \begin{eqnarray}
     qa'q''-q'(aq'') &=& q'', \nonumber
  \\ aq'a''-a'(aq'') &=& a \nonumber
  \end{eqnarray}
  we get
  \[k=qa''-aq''=\frac{q''+a''+q+a}{q'+a'}=r.\]
  We will use this and previous results to carry out some
  computations. First, we compute
  $N_{Q,\text{odd}}(1)$.
  \begin{eqnarray}
     & &\#\left\{\frac{a}{q}<\frac{a'}{q'} \text{
     consecutive in }\mathcal{F}_Q : q \text{ odd}, q'
     \text{ odd}\right\} \nonumber
  \\ & & \qquad\qquad =\#\{ (q,q') : 1\leq q, q'\leq Q,
     q+q'>Q, \nonumber
  \\ & & \qquad\qquad\qquad\qquad \, \mathop{\mathrm
     {GCD}}(q,q')=1, q \text{ odd}, q' \text{ odd}\}
     \nonumber
  \\ & & \qquad\qquad =\frac{Q^2}{\pi ^2}+O(Q\log Q),
     \nonumber
  \end{eqnarray}
  and
  \begin{eqnarray}
     & & \#\left\{ \frac{a}{q}<\frac{a'}{q'}<\frac{a''}
     {q''} \text{ consecutive in } \mathcal{F}_{Q} : q
     \text{ odd}, q' \text{ even}, q'=q+q''\right\}
     \nonumber
  \\ & & \qquad\qquad =\#\left\{\frac{a}{q}<\frac{a'}
     {q'}\text{ consecutive in } \mathcal{F}_{Q} : \right.
     \nonumber
  \\ & & \qquad\qquad\qquad\qquad\left. q \text{ odd}, q'
     \text{ even}, \left[\frac{Q+q}{q'}\right] =1
     \right\} \nonumber
  \\ & & \qquad\qquad =\#\{ (q,q') : 1\leq q,q'\leq Q,
     q+q'>Q, \, \mathop{\mathrm{GCD}}(q, q')=1, \nonumber
  \\ & &\qquad\qquad\qquad\qquad\qquad\qquad\qquad q'\leq
     Q+q<2q', q \text{ odd}, q' \text{ even}\} \nonumber
  \\ & & \qquad\qquad =\#\left\{ (q, q')\in T_{1,Q} : q
     \text{ odd}, q' \text{ even},\, \mathop{\mathrm{GCD}}(q,q')=1
     \right\} \nonumber
  \\ & & \qquad\qquad =\frac{2}{\pi ^2}\mathop{\mathrm{Area}}(T
     _{1,Q})+O(Q\log Q) \nonumber
  \\ & & \qquad\qquad =\frac{Q^2}{3\pi ^2}+O(Q\log Q).
     \nonumber
  \end{eqnarray}
  So we have
  \begin{eqnarray}
     N_{Q,\text{odd}}(1) &=& \#\left\{\frac{a}{q}<\frac{a'}
     {q'} \text{ consecutive in } \mathcal{F}_Q : q
     \text{ odd}, q' \text{ odd}\right\}  \nonumber
  \\ & & \quad + \#\left\{ \frac{a}{q}<\frac{a'}
     {q'}<\frac{a''}{q''} \text{ consecutive in } \mathcal
     {F}_{Q} : \right. \nonumber
  \\ & & \qquad\qquad\qquad\qquad q \text{ odd}, q'
     \text{ even}, q'=q+q''\bigg\} \nonumber
  \\ &=& \frac{2}{3}\cdot\frac{2Q^2}{\pi ^2}+O(Q\log Q).
     \nonumber
  \end{eqnarray}
  Now, for
  $k>1$,
  \begin{eqnarray}
     N_{Q,\text{odd}}(k)&=&\#\bigg\{ \frac{a}{q}<\frac{a'}
     {q'}<\frac{a''}{q''} \text{ consecutive in } \mathcal
     {F}_{Q} : \nonumber
  \\ & & \qquad\qquad\qquad\qquad q \text{ odd}, q'
     \text{ even}, q+q''=kq'\bigg\} \nonumber
  \\ &=& \#\left\{\frac{a}{q}<\frac{a'}{q'} \text{
     consecutive in }\mathcal{F}_{Q} : q \text{ odd}, q'
     \text{ even}, \left[\frac{Q+q}{q'}\right] =k\right\}
     \nonumber
  \\ &=& \#\{ (q, q') : 1\leq q,q'\leq Q, q+q'>Q,\, \mathop{\mathrm{GCD}}
     (q, q')=1, \nonumber
  \\ & & \qquad\qquad\qquad\qquad kq'\leq Q+q<(k+1)q', q
    \text{ odd}, q' \text{ even}\} \nonumber
  \\ &=& \#\left\{ (q, q')\in T_{k,Q} :\, \mathop{\mathrm
     {GCD}}(q,q')=1,  q \text{ odd}, q' \text{
     even}\right\} \nonumber
  \\ &=& \frac{2}{\pi ^2}\mathop{\mathrm{Area}}(T
     _{k,Q})+O(Q\log Q) \nonumber
  \\ &=& \left( \frac{2Q^2}{\pi ^2}\right) \left( \frac{4}
     {k(k+1)(k+2)}\right) +O(Q\log Q). \nonumber
  \end{eqnarray}
  We formulate this as a theorem.
  \begin{theorem}
  For any positive integer $k$,
  \[N_{Q,\text{odd}}(k)= \left( \frac{2Q^2}{\pi ^2}\right) \left( \frac{4}
     {k(k+1)(k+2)}\right) +O(Q\log Q). \]
  \end{theorem}
  Putting this together with our approximation for
  $\#\mathcal{F}_{Q,\text{odd}}$,
  we have proved the following result.
  \begin{corollary}
  For any positive integer
  $k$,
  \[\rho _{\text{odd}}(k)=\lim_{Q\to \infty}\frac{N
  _{Q,\text{odd}}(k)}{\#\mathcal{F}_{Q,\text{odd}}}=\frac
  {4}{k(k+1)(k+2)}.\]
  \end{corollary}

  \section{Intermediate Lemmas and Computations}
  In the final two sections, we will generalize the results
  of the previous sections by considering the Farey
  fractions of order
  $Q$
  with odd denominators which are contained in an interval
  $[\alpha, \beta]\subseteq [0,1]$.
  For
  $k\in\mathbb{N}$,
  we define
  \begin{eqnarray}
     \mathcal{F}_{Q, [\alpha, \beta]} &=& \left\{ \frac{a}
     {q}\in\mathcal{F}_{Q} : \frac{a}{q}\in [\alpha,
     \beta]\right\}, \nonumber
  \\ \mathcal{F}_{Q,\text{odd},[\alpha, \beta]} &=&
     \left\{\frac{a}{q}\in\mathcal{F}_{Q,\text{odd}} : \frac{a}
     {q}\in [\alpha, \beta]\right\}, \nonumber
  \\ N_{Q,\text{odd},[\alpha , \beta]}(k) &=&
     \#\left\{\frac{a}{q}<\frac{a'}{q'} \text{
     consecutive in } \mathcal{F}_{Q,\text{odd},[\alpha
     ,\beta]} : qa'-aq'=k\right\}, \nonumber
  \end{eqnarray}
  and
  \begin{eqnarray}
     \rho _{\text{odd},[\alpha, \beta]}(k) &=& \lim
     _{Q\to\infty} \frac{N_{Q,\text{odd},[\alpha
     ,\beta]}(k)}{\#\mathcal{F}_{Q,\text{odd},[\alpha
     ,\beta]}}, \nonumber
  \end{eqnarray}
  assuming that the above limit exists. First, we will
  compute
  $\#\mathcal{F}_{Q,\text{odd},[\alpha, \beta]}$.
  Start by defining the region
  \[\Omega = \left\{ (x,y)\in \mathbb{R}^2 : 0<x\leq Q,
  \alpha x\leq y\leq \beta x\right\}.\]
  It is easy to find that
  $\mathop{\mathrm{Area}}(\Omega )=(\beta -\alpha )Q^2/2$.
  Using this fact, we have
  \begin{eqnarray}
     \#\mathcal{F}_{Q,\text{odd},[\alpha, \beta]}
     &=& \#\{ (q,a) : 1\leq q\leq Q, \alpha\leq \frac{a}
     {q}\leq\beta, \, \mathop{\mathrm{GCD}}(a,q)=1, q \text{ odd}\}
     \nonumber
  \\ &=& \#\{ (x,y)\in(\mathbb{Z}^2\cap\Omega ): \,
     \mathop{\mathrm{GCD}}(x,y)=1, x \text{ odd}\}
     \nonumber
  \\ &=& N_{\text{odd}}(\Omega ) \nonumber
  \\ &=& \frac{4}{\pi ^2}\mathop{\mathrm{Area}}(\Omega )+O(Q\log Q) \nonumber
  \\ &=& \frac{2(\beta -\alpha )Q^2}{\pi ^2}+O(Q\log
     Q). \nonumber
  \end{eqnarray}
  Next, using well known estimates for Kloosterman sums
  ([8],[3]) one obtains the following lemma (see also
  [2, Lemma 1.7]).
  \begin{lemma}
  Given a positive integer
  $q$,
  let
  $\mathcal{I}, \mathcal{J}\subseteq\{0,1,\ldots ,q-1\}$
  be sets of consecutive integers. For any positive integer
  $a$
  such that
  $\mathop{\mathrm{GCD}}(a,q)=1$
  define
  \[N_{q}(\mathcal{I}, \mathcal{J}, a)=\#\{(x,y) :
  x\in\mathcal{I}, y\in\mathcal{J}, xy\equiv a\mod q\}.\]
  Then for all $\epsilon >0$ we have
  \[N_{q}(\mathcal{I}, \mathcal{J}, a)=\frac
  {\varphi (q)|\mathcal{I}||\mathcal{J}|}
  {q^2}+O_\epsilon (q^{1/2+\epsilon}).\]
  \end{lemma}
  This lemma will be used repeatedly in the following
  section. A second lemma which we will borrow is the
  following (for a proof, see [2, Lemma 2.3]).
  \begin{lemma}
  For any function
  $f(q)$
  which is
  $C^1$
  on
  $[1,Q]$
  we have:
  \begin{eqnarray}
     \sum_{q=1}^Q\frac{\varphi (q)}{q}f(q)=\frac{6}
     {\pi ^2}\int_1^Q f(q)dq+O\left(\| f\| _\infty  \log
     Q+\log Q\int_1^Q |f'(q)|dq\right). \nonumber
  \end{eqnarray}
  \end{lemma}
  It will be helpful to use Lemma 3 to compute two sums
  before proceeding to the final section. Define our first
  sum as
  \[F(Q)=\sum_{\substack{q=1 \\ q\text{ odd}}}^Q \varphi
  (q).\]
  By Lemma 3 we have
  \begin{eqnarray}
     \Phi (Q)=\sum_{q=1}^Q \varphi (q)=\frac{3Q^2}
     {\pi ^2}+O(Q\log Q).\nonumber
  \end{eqnarray}
  Using this result, we have
  \begin{eqnarray}
     \Phi (Q)-F(Q) &=& \sum_{1\leq q\leq Q/2}\varphi (2q)
     \nonumber
  \\ &=& \sum_{\substack{1\leq q\leq Q/2\\q\text{
     odd}}}\varphi(q) +2\sum_{\substack{1\leq q\leq
     Q/2\\q\text{ even}}}\varphi (q) \nonumber
  \\ &=& F\left( \frac{Q}{2}\right) +2\left(
     \Phi\left(\frac{Q}{2}\right) -F\left(\frac{Q}
     {2}\right)\right) \nonumber
  \\ &=& 2\Phi\left(\frac{Q}{2}\right)-F\left(\frac{Q}
     {2}\right), \nonumber
  \end{eqnarray}
  and therefore
  \begin{eqnarray}
     F(Q)-F\left(\frac{Q}{2}\right) &=& \Phi
     (Q)-2\Phi\left(\frac{Q}{2}\right) \nonumber
  \\ &=& \frac{3Q^2}{2\pi ^2}+O(Q\log Q). \nonumber
  \end{eqnarray}
  To find
  $F(Q)$
  explicitly, write
  \begin{eqnarray}
     F(Q) &=& \sum_{n=0}^{\lfloor\log_2  Q\rfloor+1}
     \left( F\left(\frac{Q}{2^n}\right) -F\left(\frac{Q}
     {2^{n+1}}\right)\right)
  \\ &=& \sum_{n=0}^{\lfloor\log_2  Q\rfloor +1}
     \left(\frac{1}{4^n}\cdot\frac{3Q^2}{2\pi ^2}
     +O(Q\log Q)\right) \nonumber
  \\ &=& \frac{3Q^2}{2\pi ^2}\sum_{n=0}^{\lfloor\log_2
     Q\rfloor +1}\frac{1}{4^n}+O(Q\log^2 Q) \nonumber
  \\ &=& \frac{3Q^2}{2\pi ^2}\cdot\frac{4}{3}+O(Q\log^2 Q)
     \nonumber
  \\ &=& \frac{2Q^2}{\pi ^2}+O(Q\log^2 Q). \nonumber
  \end{eqnarray}
  The second sum that we need to consider is
  \[G_1(Q)=\sum_{\substack{q=1 \\ q\text{ odd}}}^Q
  \frac{\varphi (q)}{q}.\]
  Using the same method as before, we apply Lemma 3 to
  compute
  \begin{eqnarray}
     G_2(Q)=\sum_{q=1}^Q \frac{\varphi (q)}{q}=\frac{6Q}
     {\pi ^2}+O(\log Q).\nonumber
  \end{eqnarray}
  Then we have
  \begin{eqnarray}
     G_1(Q)-\frac{1}{2}G_1\left( \frac{Q}
     {2}\right) &=& G_2(Q)-G_2\left( \frac{Q}{2}\right)
     \nonumber
  \\ &=& \frac{3Q}{\pi ^2}+O(\log Q), \nonumber
  \end{eqnarray}
  and using a similar technique as in the last sum produces
  \begin{eqnarray}
     G_1(Q) &=& \frac{4Q}{\pi ^2}+O(\log^2 Q). \nonumber
  \end{eqnarray}

  \section{Asymptotic Results in Subintervals of $[0,1]$}
  Now we will consider the values of the numbers and the
  limits defined at the beginning of the previous section.
  The same arguments used to sort out the possible cases at
  the beginning of Section 4 apply here. The main
  differences with the computations here are the use of
  the implication
  \[qa'-aq'=1 \Rightarrow a\equiv -\overline{q'}\mod q \]
  (where
  $q'\overline{q'}\equiv 1\mod q$)
  and the fact that if
  $a/q<a'/q'$
  are consecutive in
  $\mathcal{F}_Q$,
  then the following statement is true, except when
  $a/q$
  is the largest element of
  $\mathcal{F}_{Q,[\alpha ,\beta]}$:
  \begin{eqnarray}
     \alpha\leq \frac{a}{q}<\frac{a'}{q'}\leq\beta
     &\text{if and only if} & a\in [\alpha q, \beta q].
     \nonumber
  \end{eqnarray}
  Arguing as in Section 4, we begin with
  $N_{Q,\text{odd},[\alpha , \beta]}(1)$
  as a special case:
  \begin{align}
     \# & \left\{\frac{a}{q}<\frac{a'}{q'} \text{
     consecutive in }\mathcal{F}_{Q,[\alpha , \beta]} : q
     \text{ odd}, q' \text{ odd}\right\} \nonumber
  \\ =&\#\bigg\{ (q,q') : 1\leq q, q'\leq Q, q+q'>Q, \,
     \mathop{\mathrm{GCD}}(q,q')=1, \nonumber
  \\ &\qquad\qquad\qquad\qquad\qquad q \text{ odd}, q'
     \text{ odd}, \alpha\leq\frac{a}{q}<\frac{a'}
     {q'}\leq\beta\bigg\} \nonumber
  \\ =&\#\left\{ (q,q') : 1\leq q, q'\leq Q, q+q'>Q,
     \, \mathop{\mathrm{GCD}}(q,q')=1, \right. \nonumber
  \\ &\qquad\qquad\qquad\qquad\qquad\left. q \text{ odd},
     q' \text{ odd},-\overline{q'}\mod q\in[\alpha
     q,\beta q]\right\} +O(1). \nonumber
  \end{align}
  Making the substitution
  $x=q', y=-\overline{q'}, 2z=q'$,
  the last expression can be written as
  \begin{eqnarray}
     &=& \sum_{\substack{q=1 \\ q\text{ odd}}}^Q
     \#\{(x,y) : x\in [Q-q, Q], y\in [\alpha q, \beta q],
     xy\equiv -1\mod q\} \nonumber
  \\ & & -\sum_{\substack{q=1 \\ q\text{ odd}}}^Q
     \#\left\{ (z,y) : z\in \left[\frac{Q-q}{2}, \frac{Q}
     {2}\right] , y\in [\alpha q, \beta q], zy\equiv
     -\overline{2}\mod q\right\} \nonumber
  \\ & & +\quad O(1). \nonumber
  \end{eqnarray}
  Choosing any $\epsilon >0$ and applying Lemma 2 to approximate the numbers inside these sums, this becomes
  \begin{eqnarray}
     &=& \sum_{\substack{q=1 \\ q\text{ odd}}}^Q \left(
     \varphi (q)(\beta -\alpha)-\varphi (q)\frac{\beta
     -\alpha}{2}+O_{\epsilon}(q^{1/2+\epsilon})\right)
     +O(1) \nonumber
  \\ &=& \left(\frac{\beta -\alpha}{2}\sum_{\substack
     {q=1\\ q\text{ odd}}}^Q \varphi (q)\right)
     +O_{\epsilon}(Q^{3/2+\epsilon}). \nonumber
  \end{eqnarray}
  Using (1), we rewrite the preceding expression as
  \begin{eqnarray}
     &=& \frac{\beta -\alpha}{2}\left(\frac{2Q^2}{\pi
     ^2}+O(Q\log^2 Q) \right) +O_{\epsilon}(Q
     ^{3/2+\epsilon}) \nonumber
  \\ &=& \frac{(\beta -\alpha)Q^2}{\pi ^2}+O_{\epsilon}(Q
     ^{3/2+\epsilon}). \nonumber
  \end{eqnarray}
  Next we need to compute the number
  \begin{align}
     \# &\left\{ \frac{a}{q}<\frac{a'}{q'} \text{
     consecutive in }\mathcal{F}_Q : q \text{ odd}, q'
     \text{ even}, \left[ \frac{Q+q}{q'}\right]=1, a\in
     [\alpha q, \beta q]\right\} \nonumber
  \\ & =\#\{ (q,q')\in T_{1,Q} : q \text{ odd}, q' \text{
     even}, \mathop{\mathrm{GCD}}(q,q')=1, -\overline{q'}\mod q\in[\alpha
     q,\beta q]\}. \nonumber
  \end{align}
  Substituting
  $2x=q', y=-\overline{q'}$
  we write this number as
  \begin{eqnarray}
     &=& \sum_{\substack{1\leq q\leq Q/3\\ q\text{
     odd}}}\#\left\{ (x,y) : x\in \left[\frac{Q-q}{2},
     \frac{Q}{2}\right] , y\in [\alpha q, \beta q],
     xy\equiv -\overline{2}\mod q\right\} \nonumber
  \\ & & +\sum_{ \substack{Q/3<q\leq Q\\ q\text{
     odd}}}\#\left\{ (x,y) : x\in \left[\frac{Q+q}{4},
     \frac{Q}{2}\right] , y\in [\alpha q, \beta q],
     xy\equiv -\overline{2}\mod q\right\}. \nonumber
  \end{eqnarray}
  For notational purposes, define the operators
  \begin{eqnarray}
     \sideset{}{_1}{\sum} &=& \sum_{\substack{1\leq
     q\leq Q/3\\ q\text{ odd}}} \nonumber
  \end{eqnarray}
  and
  \begin{eqnarray}
     \sideset{}{_2}{\sum} &=& \sum_{\substack{Q/3<
     q\leq Q\\ q\text{ odd}}}=\sum_{\substack{1\leq
     q\leq Q\\ q\text{ odd}}}-\sum_{\substack{1\leq
     q\leq Q/3\\ q\text{ odd}}}, \nonumber
  \end{eqnarray}
  then apply Lemma 2 and rewrite the above sums as
  \begin{eqnarray}
     &=& \sideset{}{_1}{\sum} \left( \frac{(\beta
     -\alpha)\varphi (q)}{2} +O_{\epsilon }(q
     ^{1/2+\epsilon })\right) \nonumber
  \\ & & +\sideset{}{_2}{\sum} \left(\frac{\varphi (q)}{q}\cdot\frac
     {(\beta -\alpha )(Q-q)}{4}+O_{\epsilon }(q
     ^{1/2+\epsilon })\right) \nonumber
  \\ &=& \left(\frac{\beta -\alpha}{2}\sideset{}{_1}{\sum} \varphi
     (q)\right) +\left( \frac{(\beta -\alpha)Q}{4}\sideset
     {}{_2}{\sum} \frac{\varphi (q)}{q}\right)
  \\ & & -\left(\frac{\beta -\alpha}{4}\sideset{}{_2}{\sum} \varphi
     (q)\right) +O_{\epsilon }(Q^{3/2+\epsilon }).
     \nonumber
  \end{eqnarray}
  Using (1) and (2), we compute
  \begin{eqnarray}
     \sideset{}{_1}{\sum} \varphi (q) &=& \frac{2Q^2}{9\pi
     ^2}+O(Q\log^2 Q), \nonumber
  \\ \sideset{}{_2}{\sum} \varphi (q) &=& \frac{2}{\pi ^2}\left( Q
     ^2-\frac{Q^2}{9}\right) +O(Q\log^2 Q) \nonumber
  \\ &=& \frac{16Q^2}{9\pi ^2} +O(Q\log^2 Q), \nonumber
  \end{eqnarray}
  and
  \begin{eqnarray}
     \sideset{}{_2}{\sum} \frac{\varphi (q)}{q} &=&
     \frac{4}{\pi ^2}\left( Q-\frac{Q}{3}\right) +O(\log^2
     Q)\nonumber
  \\ &=& \frac{8Q}{3\pi ^2} +O(\log^2 Q). \nonumber
  \end{eqnarray}
  Plugging these values into (3), we find that
  \begin{align}
     \# &\left\{ \frac{a}{q}<\frac{a'}{q'} \text{
     consecutive in }\mathcal{F}_Q : q \text{ odd}, q'
     \text{ even}, \left[ \frac{Q+q}{q'}\right]=1, a\in
     [\alpha q,\beta q]\right\} \nonumber
  \\ &= \frac{(\beta -\alpha)Q^2}{9\pi ^2}+\frac
     {2(\beta -\alpha)Q^2}{3\pi ^2}-\frac{4(\beta -\alpha)Q
     ^2}{9\pi ^2}+O(Q^{3/2+\epsilon })\nonumber
  \\ &= \frac{(\beta -\alpha)Q^2}{3\pi ^2}+O(Q
     ^{3/2+\epsilon} ).\nonumber
  \end{align}
  Now, we have
  \begin{eqnarray}
     N_{Q,\text{odd},[\alpha , \beta]}(1) &=&
     \#\left\{\frac{a}{q}<\frac{a'}{q'} \text{
     consecutive in } \mathcal{F_{Q,[\alpha
     ,\beta]}} : q \text{ odd}, q' \text{ odd}\right\}
     \nonumber
  \\ & & +\#\left\{ \frac{a}{q}<\frac{a'}{q'} \text{
     consecutive in } \mathcal{F}_Q : q \text{ odd}, q'
     \text{ even}, \right. \nonumber
  \\ & & \qquad\qquad\qquad\qquad\qquad \left. \left[ \frac{Q+q}
     {q'}\right]=1, a\in [\alpha q, \beta q]\right\}
     +O(1) \nonumber
  \\ &=& \frac{4(\beta -\alpha)Q^2}{3\pi ^2}+O_{\epsilon
     }(Q^{3/2+\epsilon }). \nonumber
  \end{eqnarray}
  Putting this together with our value for
  $\#\mathcal{F}_{Q,\text{odd},[\alpha ,\beta]}$,
  we find that
  \[\rho _{\text{odd},[\alpha ,\beta]}(1)=\frac{2}{3}.\]
  We will use similar arguments for the final computation.
  Take
  $k>1$,
  and we have
  \begin{eqnarray}
     N_{Q,\text{odd},[\alpha ,\beta]}(k) &=&
     \#\left\{ \frac{a}{q}<\frac{a'}{q'} \text{
     consecutive in }\mathcal{F}_Q : q \text{ odd}, q'
     \text{ even}, \right. \nonumber
  \\ & &\qquad\qquad\qquad\qquad \left.\left[\frac{Q+q}
     {q'}\right]=k, a\in [\alpha q, \beta q]\right\} +O(1)
     \nonumber
  \\ &=& \#\{ (q,q')\in T_{k,Q} : q \text{ odd}, q' \text
     { even}, \nonumber
  \\ & & \qquad\qquad\qquad\qquad \mathop{\mathrm
     {GCD}}(q,q')=1,-\overline{q'}\mod q\in[\alpha
     q,\beta q]\} \nonumber
  \\ & & +\quad O(1) \nonumber
  \end{eqnarray}
  Substitute
  $2x=q', y=-\overline{q'}$
  and write this number as
  \begin{eqnarray}
     &=& \sum_{\substack{(k-1)Q/(k+1)\leq q\leq
     kQ/(k+2)\\ q\text{ odd}}}\#\left\{ (x,y) : x\in
     \left[\frac{Q-q}{2}, \frac{Q+q}{2k}\right] , \right.
     \nonumber
  \\ & & \quad\qquad\qquad\qquad\qquad\qquad\qquad
     y\in\left[\alpha q,\beta q\right] , xy\equiv
     -\overline{2}\mod q\bigg\} \nonumber
  \\ & & +\sum_{\substack{kQ/(k+2)<q\leq Q\\q\text{
     odd}}}\#\left\{ (x,y) : x\in \left[\frac{Q+q}
     {2(k+1)}, \frac{Q+q}{2k}\right] , \right. \nonumber
  \\ & & \quad\qquad\qquad\qquad\qquad\qquad\qquad y\in
     \left[\alpha q,\beta q\right] , xy\equiv -\overline
     {2}\mod q\bigg\} \nonumber
  \\ & & +\quad O(1) \nonumber
  \end{eqnarray}
  Again, for notational purposes, define the operators
  \begin{eqnarray}
     \sideset{}{_3}{\sum} &=& \sum_{\substack
     {(k-1)Q/(k+1)\leq q\leq kQ/(k+2)\\ q\text{
     odd}}}=\sum_{\substack{1\leq q\leq kQ/(k+2)\\ q\text
     { odd}}}-\sum_{\substack{1\leq q<(k-1)Q/(k+1)\\
     q\text{ odd}}} \nonumber
  \end{eqnarray}
  and
  \begin{eqnarray}
     \sideset{}{_4}{\sum} &=& \sum_{\substack
     {kQ/(k+2)<q\leq Q\\ q\text{ odd}}}=\sum
     _{\substack{1\leq q\leq Q\\ q\text{ odd}}}-\sum
     _{\substack{1\leq q\leq kQ/(k+2)\\ q\text{ odd}}},
     \nonumber
  \end{eqnarray}
  and apply Lemma 2 to rewrite the above sums as
  \begin{eqnarray}
     &=& \sideset{}{_3}{\sum} \left(\frac{(\beta -\alpha)\varphi (q)}
     {q}\cdot\frac{Q+q-kQ+kq}{2k}+O_{\epsilon}(q
     ^{1/2+\epsilon})\right) \nonumber
  \\ & & +\quad\sideset{}{_4}{\sum} \left(\frac{(\beta
     -\alpha)\varphi (q)}{q}\cdot\frac{Q+q}
     {2k(k+1)} +O_{\epsilon}(q^{1/2+\epsilon})
     \right) +O(1) \nonumber
  \\ &=& \left(\frac{(\beta -\alpha)(1-k)Q}{2k}\sideset{}
     {_3}{\sum}\frac{\varphi (q)}{q}\right)+\left(\frac
     {(\beta -\alpha)(1+k)}{2k}\sideset{}{_3}{\sum}
     \varphi (q)\right)
  \\ & & + \left(\frac{(\beta -\alpha)Q}{2k(k+1)}\sideset
     {}{_4}{\sum}\frac{\varphi (q)}{q}\right)+\left(\frac
     {(\beta -\alpha)}{2k(k+1)}\sideset{}{_4}{\sum}
     \varphi (q)\right) \nonumber
  \\ & & +\quad O_{\epsilon}(Q^{3/2+\epsilon}). \nonumber
  \end{eqnarray}
  Using (1) and (2), we compute
  \begin{eqnarray}
     \sideset{}{_3}{\sum} \varphi (q) &=& \frac{2}{\pi ^2}\left(
     \left(\frac{kQ}{k+2}\right)^2-\left(\frac{(1-k)Q}
     {k+1}\right)^2\right) +O(Q\log^2 Q) \nonumber
  \\ &=& \frac{8Q^2}{\pi ^2}\cdot\frac{k^2+k-1}{(k+1)
     ^2(k+2)^2}+O(Q\log^2 Q), \nonumber
  \\ \sideset{}{_4}{\sum} \varphi (q) &=& \frac{2}{\pi ^2}\left(
     Q^2-\left(\frac{kQ}{k+2}\right) ^2\right) +O(Q\log^2
     Q)\nonumber
  \\ &=& \frac{8Q^2}{\pi ^2}\cdot\frac{k+1}{(k+2)^2}
     +O(Q\log^2 Q),\nonumber
  \\ \sideset{}{_3}{\sum} \frac{\varphi (q)}{q} &=& \frac{4}{\pi
     ^2}\left( \frac{kQ}{k+2}-\frac{(1-k)Q}{k+1}\right)
     +O(\log^2 Q)\nonumber
  \\ &=& \frac{8Q}{(k+1)(k+2)\pi ^2} +O(\log^2 Q),
     \nonumber
  \end{eqnarray}
  and
  \begin{eqnarray}
     \sideset{}{_4}{\sum} \frac{\varphi (q)}{q} &=& \frac
     {4}{\pi ^2}\left( Q-\frac{kQ}{k+2}\right) +O(\log^2
     Q) \nonumber
  \\ &=& \frac{8Q}{(k+2)\pi ^2} +O(\log^2 Q). \nonumber
  \end{eqnarray}
  We use these values to rewrite (4) as
  \begin{eqnarray}
     &=& \frac{(\beta -\alpha)Q}{2k}\left(\frac{8(1-k)Q}
     {(k+1)(k+2)\pi ^2}+\frac{8Q}{(k+1)(k+2)\pi ^2}\right)
     \nonumber
  \\ & & \qquad +\frac{\beta -\alpha}{2k}\left(\frac
     {8(1+k)Q^2}{\pi ^2}\cdot\frac{k^2+k-1}{(k+1)^2(k+2)
     ^2}+\frac{8Q^2}{(k+1)\pi ^2}\cdot\frac{k+1}{(k+2)
     ^2}\right) \nonumber
  \\ & & \qquad + O_{\epsilon}(Q^{3/2+\epsilon}) \nonumber
  \\ &=& \frac{2(\beta -\alpha)Q^2}{\pi ^2}\cdot\frac{4}
     {k(k+1)(k+2)}+O_{\epsilon}(Q^{3/2+\epsilon}).
     \nonumber
  \end{eqnarray}
  This proves the following theorem.
  \begin{theorem}
  For any positive integer
  $k$, any interval
  $[\alpha ,\beta ]\subseteq [0,1]$, and any $\epsilon >0,$
  \[N_{Q,\text{odd},[\alpha ,\beta]}(k)= \frac{2(\beta -\alpha)Q^2}{\pi^2}\cdot\frac{4}{k(k+1)(k+2)}+O_{\epsilon}(Q^{3/2+\epsilon}).\]
  \end{theorem}
  Putting this together with our approximation for
  $\#\mathcal{F}_{Q,\text{odd},[\alpha
  ,\beta]}$, we have proved the following result.
  \begin{corollary}
  For any positive integer
  $k$
  and any interval
  $[\alpha , \beta]\subseteq [0,1]$,
  \[\rho _{\text{odd},[\alpha ,\beta]}(k)= \lim_{Q\to\infty} \frac{N_{Q,\text{odd},[\alpha
  ,\beta]}(k)}{\#\mathcal{F}_{Q,\text{odd},[\alpha
  ,\beta]}}=\frac{4}
  {k(k+1)(k+2)}.\]
  \end{corollary}


\begin{thebibliography}{9}
  \bibitem{vA01}
  V. Augustin, F. P. Boca, C. Cobeli, A. Zaharescu,
  \emph{The h-spacing distribution between Farey points},
  Math. Proc. Cambridge Philos. Soc. 131
  (2001), no. 1, 23-38.
  \bibitem{fB00}
  F. P. Boca, C. Cobeli, A. Zaharescu,
  \emph{Distribution of lattice points visible from the
  origin},
  Commun. Math. Phys. 213 (2000), 433-470.
  \bibitem{tE61}
  T. Estermann,
  \emph{On Kloosterman's sums},
  Mathematika 8 (1961), 83-86.
  \bibitem{jF24}
  J. Franel,
  \emph{Les suites de Farey et le probl\` eme des
  nombres premiers},
  Gottinger Nachr. (1924), 191-201.
  \bibitem{rH84}
  R. R. Hall, G. Tenenbaum,
  \emph{On consecutive Farey arcs},
  Acta Arith. 44 (1984), 397-405.
  \bibitem{gH38}
  G. H. Hardy, E. M. Wright,
  \emph{An Introduction to the Theory of Numbers},
  Clarendon Press, Oxford, 1938 (fourth edition 1960).
  \bibitem{aH02}
  A. Haynes
  \emph{A Note on Farey Fractions With Odd Denominators},
  (to appear).
  \bibitem{eL24}
  E. Landau,
  \emph{Bemerkungen zu der vorstehenden Abhandlung
    von Herrn Franel},
  Gottinger Nachr. (1924), 202-206.
  \bibitem{aW48}
  A. Weil,
  \emph{On some exponential sums},
  Proc. Nat. Acad. U.S.A. 34 (1948), 204-207.
  \end{thebibliography}
\end{document}